\documentclass[11pt,reqno]{amsart}
\usepackage[T1]{fontenc}
\usepackage{graphicx}

\usepackage{color}
\definecolor{MyLinkColor}{rgb}{0,0,0.4}

\newcommand{\R}{{\mathbb R}}
\newcommand{\E}{{\mathcal E}}

\newcommand{\N}{{\mathbb N}}

\newcommand{\cI}{\mathcal{I}}

\newcommand{\cP}{\mathcal{P}}

\newcommand{\ov}{\overline}
\newcommand{\p}{\partial}

\newcommand{\e}{\varepsilon}

\newcommand{\id}{\mathop{\rm id}\nolimits}
\newcommand{\supp}{\mathop{\rm supp}\nolimits}

\newtheorem{thm}{Theorem}[section]

\newtheorem{lemma}[thm]{Lemma}
\newtheorem{cor}[thm]{Corollary}
\theoremstyle{remark} 

\numberwithin{equation}{section}

\title[Global weak solutions for a fourth order thin film system]{Non-negative global weak solutions for a degenerated parabolic system approximating the
two-phase Stokes problem}
\thanks{}

\author[J. Escher]{Joachim Escher}
\address{Institut f{\"u}r Angewandte Mathematik, Leibniz Universit{\"a}t Hannover, Welfengarten~1, 30167 Hannover, Germany.}
\email{escher@ifam.uni-hannover.de}

\author[B.--V. Matioc]{Bogdan--Vasile Matioc}
\address{Institut f\" ur Mathematik, Universit\" at Wien, Nordbergstra{\ss}e 15,
1090 Wien, {\"O}sterreich}
\email{bogdan-vasile.matioc@univie.ac.at}


\usepackage[colorlinks=true,linkcolor=MyLinkColor,citecolor=MyLinkColor]{
hyperref} 

\begin{document}

\begin{abstract}
We establish the existence of non-negative  global  weak solutions for a
strongly couple degenerated parabolic system which was obtained as an 
approximation of the two-phase Stokes problem driven solely by capillary forces.
Moreover, the system under consideration may be viewed as a two-phase generalization of the classical Thin Film equation.
\end{abstract}

\maketitle

\section{Introduction and the main result}\label{S:0}

In this paper we study the following system of one-dimensional degenerated parabolic equations 
 \begin{subequations}\label{P}
\begin{equation}\label{S}
\left\{
\begin{array}{llll}
\p_t f=\displaystyle-\p_x\left( f^3\p_x^3f+\frac{R}{2}\left(2f^3+3f^2g\right)\p_x^3(f+g)\right),\\[2ex]
\p_tg=\displaystyle-\p_x   \left(\frac{3}{2}  f^2g\p_x^3f+\frac{R}{2}\left(2\mu g^3+3f^2g+6fg^2\right)\p_x^3(f+g)\right)
\end{array}
\right.
\end{equation}
for $ (t,x)\in (0,\infty)\times \cI$, where $\cI:=(0,L)$ for some $L>0$. 
This system models
the motion of the interfaces of two immiscible  fluid layers of width $f$ and $g$,
respectively.  
The layer of width $f$ is located on a impermeable horizontal bottom, identified
with the line $y=0,$ and the layer of width $g$ is located on top of the first
one. 
The system \eqref{S} has been recently derived  in \cite{EMM12b} as a thin film
approximation of the two-phase Stokes problem when capillary is the sole driving
mechanism. 
The constants $R$ and $\mu,$ which are both assumed to be positive, are
determined by material properties of the fluids and are given by
\[
\mu:=\frac{\mu_f}{\mu_g}\qquad\text{and}\qquad R:=\frac{\gamma_g}{ \gamma_f},
\]
with $\mu_f$ [resp. $\mu_g$] denoting the viscosity coefficient of the fluid
layer of width $f$ [resp. $g$].
Moreover,   $\gamma_f$ [resp. $\gamma_g$] is the surface tension coefficient at
the interface $y=f(t,x)$ [resp. $y=f(t,x)+g(t,x)$].
The system \eqref{S}  is supplemented by the initial conditions
\begin{equation}\label{eq:bc1}
 f(0)=f_0,\qquad g(0)=g_0 \qquad  \text{in $\cI$,}
 \end{equation}
whereby $f_0$ and $g_0$ are assumed to be known non-negative functions, and by no-flux
boundary conditions 
\begin{equation}\label{eq:bc2}
 \p_xf=\p_xg=\p^3_xf=\p_x^3 g=0,\qquad x=0, L.
 \end{equation}
\end{subequations}

Let us first observe that if one of the fluid layer has constant zero width,
then the system \eqref{S} becomes, up to a scaling factor,  the well-known
Thin Film equation
\begin{equation}\label{eq:TF}
\p_t h=\p_x^3(h^m\p_xh), \qquad m>0,
\end{equation}   
with $m=3$.
The  theory of existence of weak solutions for the Thin Film equation  \eqref{eq:TF}  is well-established nowadays, cf.   \cite{B93, BF90,   BP94, BP96, BG98, T07}, to mention just some of the most important contributions. 
We emphasize that it has been rigorously  proved in \cite{ GP08,  MP12} (see also \cite{GO03}) that  suitably
rescaled solutions of the Stokes and of the Hele-Shaw problem converge towards
corresponding solutions of the equation \eqref{eq:TF} with $m=3$ and  $m=1$, 
respectively. 
Compared to the Thin Film equation \eqref{eq:TF}, the system \eqref{P} is much
more complex because it exhibits a strong coupling as  both equations contain
highest order derivatives of all the unknowns.  
There are also two sources of degeneracy because both interfaces may vanish on 
subsets of the interval $\cI.$

It is worth mentioning that there exists also a two-phase generalization
corresponding to the  thin-film equation \eqref{eq:TF} with $m=1$, which has
been derived in \cite{EMM12} for flows with  capillary and gravity effects.
The resulting    system, which  has been  investigated in 
\cite{EM12x, LM12xx, BM12} in the presence of capillary and in \cite{ELM11,
LM12x}  for flows driven only by gravity, appears as the thin layer
approximation of the the two-phase Muskat problem.
Compared to \eqref{S}, the parabolic system obtained in \cite{EMM12} has much
more structure: there are two energy functionals available and, furthermore,  
the system can be interpreted as  a gradient flow for the {$L_2$}-Wasserstein distance
in the space of probability measures with finite second moment.
There are not many systems of  equations which enjoy this nice geometric property. 
We mention that the parabolic-parabolic  Keller-Segel system which has a mixed
$L_2$-Wasserstein gradient flow structure \cite{BLxx}.
As far as we know, the two-phase generalization of the thin-film equation
\eqref{eq:TF} with $m\notin\{1,3\}$ has not been discovered yet.

When studying the problem \eqref{P}, one has to rely only on the energy
functional 
\[
\E(f,g):=\frac{1}{2}\int_\cI |\p_x f|^2+R| \p_x(f+g)|^2\, d x,
\]
a fact which forces us to introduce here a weaker notion of solutions than that in \cite{LM12x, BM12}.
Indeed, it is not difficult to see that the functional $\E$ decreases along
smooth solutions  of  \eqref{P}, as we have
\begin{align}
\frac{d\E(f,g)}{dt}=&\int_\cI ((1+R)\p_xf+R\p_xg)\p_x (\p_tf)+R(\p_x(f+g))\p_x(\p_tg)\, d x\nonumber\\=&-\int_\cI ((1+R)\p_x^2f+R\p_x^2g) \p_tf+R( \p_x^2(f+g))\p_tg\, d x\nonumber\\
=&-\int_\cI((1+R)\p_x^3f+R\p_x^3g)\left(f^3\p_x^3f+\frac{R}{2}\left(2f^3+3f^2g\right)\p_x^3(f+g)\right)\, d x\nonumber\\
&-\int_\cI R \p_x^3(f+g)   \left(\frac{3}{2}  f^2g\p_x^3f+\frac{R}{2}\left(2\mu g^3+3f^2g+6fg^2\right)\p_x^3(f+g)\right)\, d x\nonumber\\
=&-\mu R^2 \int_\cI g^3\left|\p_x^3(f+g)\right|^2 \, dx-\int_\cI f\left|f\p_x^3f +\frac{R}{2}(2f+3g)\p_x^3(f+g)\right|^2\, dx\nonumber\\
&-\frac{3R^2}{4}\int_\cI fg^2\left|\p_x^3(f+g)\right|^2\, dx.\label{E1}
\end{align}
Introducing  a suitable regularized version of \eqref{S},  we construct first,
for non-negative initial data,  globally defined  
Galerkin approximations which are found to converge towards weak solutions of the
approximating systems.
On the other hand,  the energy functional $\E$ may be used to obtain  estimates
for the solutions of the  approximating systems and we obtain sufficient
information  to prove that they converge towards weak solutions of the original
problem \eqref{P}.
Though it is a priori not clear whether the weak solutions of the approximating systems
are non-negative, we prove that the  weak solutions of  \eqref{P} have this
property.   
This differs from the framework of thin fluid models with capillary effects and
insoluble surfactant \cite{EHLW12, GW06}  where the approximating regularized problems may be constructed such  that starting from
non-negative initial data the associated weak solutions  are
also non-negative.

Given $T\in (0,\infty],$ let $Q_T:=(0,T)\times\cI.$ 
The main result of this paper is the following theorem, 
establishing the existence of global and non-negative weak solutions for the
problem \eqref{P}
that  start from arbitrary non-negative initial data.
\begin{thm}\label{T:M} 
Let $f_0,g_0\in H^1(\cI)$ be two non-negative functions. 
 Then, there exists at least a  weak global solutions $(f, g)$ of problem
\eqref{P} with the following properties:
\begin{itemize}
\item[(a)] $ f,g\in L_\infty(0,T;H^1(\cI))\cap \left(\cap_{\alpha\in[0,1/2)}C([0,T], C^\alpha(\ov\cI))\right) $ for all $T>0$;\medskip
\item[(b)] $ (f, g)(0)=(f_0,g_0)$ and $f\geq0,$ $ g\geq0$ in $(0,\infty)\times\cI$; \medskip
\item[(c)] the mass of the fluids is conserved, that is for every $t>0$
\[
 \|f(t) \|_{L_1}=\|f_0\|_{L_1} \qquad\text{and}\qquad\|g(t)\|_{L_1}=\|g_0\|_{L_1};
\]

\item[(d)] defining for every $T>0 $ the sets
\begin{align*}
 \cP_f:=\{(t,x)\in Q_T\,:\, f(t,x)>0)\}, \qquad \cP_g:=\{(t,x)\in Q_T\,:\, g(t,x)>0)\},
\end{align*}
we have $\p_x^3f, \p_x^3g\in L_{2,loc}(\cP_f\cap \cP_g) $ and there exists functions $j_f, j_g, j_{f,g}\in L_2(Q_T)$ with
\begin{equation*}
\left.
\begin{array}{llll}
j_f=f^{1/2}\left(f\p_x^3f+\frac{R}{2} (2f+3fg)\p_x^3(f+g)\right),\\[2ex]
j_{g}=g^{3/2}\p_x^3(f+g), \quad  j_{f,g}=f^{1/2}g\p_x^3(f+g)
\end{array}
\right.  \qquad\text{a.e. in $\cP_f\cap \cP_g$,}
\end{equation*}
and such that  
\[\text{ $H_f:=f^{3/2}j_f$,\quad $ H_g:=\mu Rg^{3/2}j_g+\frac{3R}{4}f^{1/2}gj_{f,g}+\frac{3}{2}f^{1/2}gj_f$ \qquad belong to $L_2(Q_T)$,}\]
and
\begin{align}
&\int_{Q_T}f  \p_t\xi\, dxdt+\int_{Q_T}H_f \p_x\xi\, dxdt={ \int_\cI f(T,x)  \xi(T,x )\, dx}-\int_\cI f_0\xi(0,x)\, dx, \label{I1}\\[1ex]
&\int_{Q_T}g  \p_t\xi\, dx dt+\int_{Q_T} H_g\p_x\xi\, dxdt={ \int_\cI g(T,x)  \xi(T,x )\, dx}-\int_\cI g_0\xi(0,x)\, dx \label{I2}
\end{align}
for all    $\xi\in C^\infty(\ov Q_T)$;
\item[(e)] the energy inequality 
\begin{align}
 \E(f(T),g(T))&+\int_{\cP_f\cap\cP_g}  \mu R^2 \left|j_g\right|^2 +\frac{3R^2}{4} \left|j_{f,g}\right|^2+\left|j_f \right|^2dxdt\leq \E(f_0,g_0)\label{Energy}
 \end{align}
is satisfied for almost all $T>0.$
\end{itemize}
\end{thm}

We emphasize that due to the lack of regularity of the weak solutions $(f,g)$ found in Theorem \ref{T:1},
which is mainly due to the strong coupling of the system \eqref{S}, we can identify the 
 function $H_f$ only in    $\cP_g$ and $H_g$ only in $\cP_f$:
 \begin{align*} 
H_f=&\left(f^3\p_x^3f+\frac{R}{2}\left(2f^3+3f^2g\right)\p_x^3(f+g)\right)\mathbf{1}_{(0,\infty)}(f)\qquad \text{a.e. in $\cP_g$},\\[1ex]
H_g=&\left(\frac{3}{2} f^2g\p_x^3f+\frac{R}{2}\left(2\mu g^3+3f^2g+6fg^2\right)\p_x^3(f+g)\right)\mathbf{1}_{(0,\infty)}(g)\qquad\text{a.e. in $\cP_f$.}
\end{align*}
Particularly, if the test function $\xi$ in \eqref{I1} satisfies additionally $\supp\p_x\xi\subset\cP_g,$ then  \eqref{I1} is exactly the  equation one obtains when multiplying the first equation by \eqref{S} by $\xi$
and intergrating by parts (similarly for \eqref{I2}).

The outline of the paper is as follows: in Section \ref{S:2} we regularize the
problem and construct global weak solutions for the approximating regularized 
systems.
Furthermore, we prove that any accumulation point of the set of approximating weak solutions has to be non-negative. 
Based upon the estimates deduced for this family of weak solutions, we prove in
Section \ref{S:3} that certain sequences of approximating weak solutions  converge, when letting the regularization parameter go to zero, towards non-negative weak solutions of the original problem \eqref{P}.

\section{The regularized approximating problems}\label{S:2}

In this section we construct a family of regularized  systems approximating in
the limit the original system \eqref{S}.
This is done in such a manner that the energy functional $\E$ still decreases  along 
solutions of the regularized system.
The advantage of such a construction is twofold: first it enables us to find
globally defined Galerkin approximations which are shown to converge towards
weak solutions of the approximating system, and secondly 
 it provides us with sufficient information in order to find accumulation points of
this family of weak solutions  which solve the problem \eqref{P} in the  weak
sense defined in Theorem \ref{T:M}.   
    
To proceed we define for every $\e\in(0,1]$, the Lipschitz continuous function
$a_\e:\R\to\R$ by the relation
\[
\text{$a_\e(s):=\e+\max\{0,s\}$ \qquad for $s\in\R$.}
\]
With this notation,  we introduce the following modified version of the original problem
\eqref{S}
 \begin{equation}\label{S1}
\left\{
\begin{array}{llll}
\p_t f_\e=\displaystyle-\p_x \left( a^3_\e(f_\e)\p_x^3f_\e+\frac{R}{2}\left(2a^3_\e(f_\e)+3a^2_\e(f_\e)a_\e(g_\e)\right)\p_x^3(f_\e+g_\e)\right),\\[2ex]
\p_tg_\e=\displaystyle-\p_x\left(  \frac{3}{2} a^2_\e(f_\e)a_\e(g_\e)\p_x^3f_\e\right.\\[2ex]
\hspace{2.2cm}\displaystyle\left.+\frac{R}{2}\left(2\mu a^3_\e(g_\e)+3a^2_\e(f_\e)a_\e(g_\e)+6a_\e(f_\e)a^2_\e(g_\e)\right)\p_x^3(f_\e+g_\e)\right),
\end{array}
\right.
\end{equation}
which is more regular than \eqref{S} in the sense that the coefficients of the
fourth-order derivatives are bounded from below by a positive constant depending
only on $\e$.
The system \eqref{S1} is supplemented by the initial and boundary conditions
\eqref{eq:bc1} and \eqref{eq:bc2}. 

The main result of this section is the following theorem, ensuring the solvability of the regularized   approximating problem \eqref{S1}, \eqref{eq:bc1}, and \eqref{eq:bc2} for any $\e\in(0,1].$ 
\begin{thm}\label{T:1} 
Let $f_0,g_0\in H^1(\cI)$ be two non-negative functions and $\e\in(0,1]$ be
fixed. 
 Then, there exists at  least a  couple of functions $(f_\e, g_\e)$ having the
following regularity
\begin{itemize}
\item $ \displaystyle f_\e,g_\e\in L_\infty(0,T;H^1(\cI))\cap L_2(0,T;H^3 (\cI))\cap\left(\cap_{\alpha\in[0,1/2)} C([0,T], C^\alpha(\ov\cI))\right),$
\item $\p_tf_\e,\p_tg_\e\in L_2(0,T; \left(H^1( \cI)\right)')$,
\end{itemize}
and  satisfying 
\begin{equation}\label{11}
\begin{aligned}
&\int_0^T\langle\p_t f_\e(t) | \xi(t) \rangle \, dt=\int_{Q_T}\left(a^3_\e(f_\e)\p_x^3f_\e   +\frac{R}{2}\left(2a^3_\e(f_\e )+3a^2_\e(f_\e )a_\e(g_\e)\right)\p_x^3(f_\e +g_\e )\right)\p_x\xi\, dx dt,\\[1ex]
&\int_0^T\langle\p_t g_\e (t)| \xi(t) \rangle \,dt=\int_{Q_T}\left(\frac{R}{2}\left(2\mu a^3_\e(g_\e)+3a^2_\e(f_\e)a_\e(g_\e)+6a_\e(f_\e)a^2_\e(g_\e)\right)\p_x^3(f_\e+g_\e)\right.\\[1ex]
&\hspace{4.6cm}\left.+\frac{3}{2}a^2_\e(f_\e)a_\e(g_\e)\p_x^3f_\e\right)\p_x\xi\, dx dt
\end{aligned}
\end{equation}
for all $T>0$ and all $\xi\in L_2(0,T;H^1(\cI)), $ whereby
$\langle\cdot|\cdot\rangle$ is the pairing between $H^1(\cI)$ and
$\left(H^1(\cI)\right)'$.
Moreover, $(f_\e,g_\e)(0)=(f_0,g_0)$,
 \begin{align}\label{12}
&\int_\cI f_\e(T)\,dx=\|f_0\|_{L_1}\quad\text{and}\quad\int_\cI
g_\e(T)\,dx=\|g_0\|_{L_1}\qquad\text{for all $T\geq0$},\\[1ex]
&\p_xf_\e(T)=\p_xg_\e(T)=0\qquad\text{at $x=0,L$ for almost all $T>0$,}\label{bb3}
\end{align}
and the energy inequality
\begin{align}
 &\E(f_\e(T),g_\e(T))+\int_{Q_T}a_\e(f_\e)
\left|a_\e(f_\e)\p_x^3f_\e+\frac{R}{2}
(2a_\e(f_\e)+3a_\e(g_\e))\p_x^3(f_\e+g_\e)\right|^2  dxdt\nonumber\\
&\phantom{=}+ \mu R^2\int_{Q_T} a_\e^3(g_\e)\left|\p_x^3(f_\e+g_\e)\right|^2dxdt+\frac{3R^2}{4}\int_{Q_T}a_\e(f_\e)a_\e^2(g_\e)\left|\p_x^3(f_\e+g_\e)\right|^2 dxdt\nonumber\\
&\leq \E(f_0,g_0)\label{13}
 \end{align}
is satisfied for almost all $T>0.$
\end{thm}

\vspace{0.2cm}
\subsection{\bf Approximations of the weak solutions of \eqref{S1} by Fourier
series expansions}
$ $ \vspace{0.2cm}

In the remaining of this section $\e\in(0,1]$ is arbitrary but fixed.
In a first step we construct Galerkin approximations for the  weak solution of
the problem determined by \eqref{S1}, \eqref{eq:bc1} and \eqref{eq:bc2}.
 Letting
\[
\phi_0:=\sqrt{1/L}\qquad\text{and} \qquad \phi_k:=\sqrt{{2}/{L}}\cos(k\pi x/L),
\ k\geq 1,
\]
 denote  the normalized eigenvectors of the operator $-\p_x^2:H^2(\cI)\to
L_2(\cI)$ which satisfy zero Neumann boundary conditions,
it is well-known that  any function which belongs to $H^1(\cI)$ can be
represented in $H^1(\cI)$ by its trigonometric series. 
Let thus $f_0, g_0$  be two non-negative functions from $H^1(\cI)$. 
For each fixed $n\in\N$,  we consider the  partial sums 
 \[
f_0^n:=\sum_{k=0}^n f_{0k}\phi_k, \quad g_0^n:=\sum_{k=0}^n g_{0k}\phi_k
\] 
of the  series expansions for the initial conditions $(f_0,g_0 )$, and  we   
seek  for continuously differentiable  functions  
  \[
f_\e^n(t,x):=\sum_{k=0}^n F_\e^k(t)\phi_k(x), \quad g_\e^n(t):=\sum_{k=0}^n
G_\e^k(t)\phi_k(t,x)
\]
for $t\geq0 $ and $x\in\ov \cI,$ 
 which solve  \eqref{S1} when testing with functions from the linear subspace
$\langle\phi_0,\ldots,\phi_n\rangle,$ and which satisfy initially
 \[
f_\e^n(0)=f_{0}^n,\qquad g_\e^n(0)=g_{0}^n.
\]
By construction the functions $(f_\e^n,g_\e^n)$ satisfy the boundary
conditions \eqref{eq:bc2} and, if we 
test both equations of \eqref{S1} with $\phi_0$, it follows at once that
$F_\e^0$ and $G_\e^0$ are constant in time, that is
\begin{equation}\label{eq:00}
F_\e^0(t)=f_{00},\qquad G_\e^0(t)=g_{00},\qquad t\geq0.
\end{equation}
Additionally, testing the system \eqref{S1} successively with
$\phi_1,\ldots,\phi_n$, it follows that the $2n-$tuple $( \mathbb{
F},\mathbb{G}):=(F_\e^1,\ldots,F_\e^n,G_\e^1,\ldots,G_\e^n)$
is the solution of the initial value problem
\begin{equation}\label{eq:IVP}
(\mathbb{ F},\mathbb{G})'=\Psi( \mathbb{ F},\mathbb{G}),\qquad (\mathbb{
F},\mathbb{G})(0)=(f_{01},\ldots,f_{0n},g_{01},\ldots,g_{0n}),
\end{equation}
 whereby the function  $\Psi:=(\Psi_1,\Psi_2):\R^{2n}\to\R^{2n}$
 is  given by 
\begin{align*}
\Psi_{1,j}(p,q)=&\sum_{k=1}^n p_k\int_0^L
a^3_\e(\Xi_f(p))\p_x^3\phi_k\p_x\phi_j\, dx\\
&+R\sum_{k=1}^n(p_k+q_k)\int_0^L
\left(a_\e^3(\Xi_f(p))+\frac{3}{2}
a_\e^2(\Xi_f(p))a_\e(\Xi_g(q))\right)\p_x^3\phi_k\p_x\phi_j\, dx,
 \end{align*}
and
\begin{align*}
\Psi_{2,j}(p,q)=&\frac{3}{2}\sum_{k=1}
^n p_k\int_0^La_\e^2(\Xi_f(p))a_\e(\Xi_g(q))\p_x^3\phi_k\p_x\phi_j\, dx\\
&+R\sum_{k=1}^n(p_k+q_k)\int_0^L \left(\mu a_\e^3(\Xi_g(q))+\frac{3}{2}a_\e^2(\Xi_f(p))a_\e(\Xi_g(q))\right.\\
&\left.\hspace{3.5cm}\phantom{\frac{3}{2}}+3a_\e(\Xi_f(p))a_\e^2(\Xi_g(q))\right)\p_x^3\phi_k\p_x\phi_j\, dx\\
 \end{align*}
for $j\in\{1,\ldots,n\}$ and $(p,q)\in\R^{2n}$. 
We used here the shorthand
\[
\Xi_f(p)=f_{00}\phi_0+\sum_{l=1}^n p_l\phi_l,\quad
\Xi_g(q)=g_{00}\phi_0+\sum_{l=1}^n q_l\phi_l
\]
 for all $p,q\in\R^n.$
 Recalling that  $a_\e$ is a Lipschitz continuous function, we deduce that
$\Psi$ is locally Lipschitz continuous in $\R^{2n},$ 
and therefore the initial value problem  \eqref{eq:IVP} possesses a unique
solution $(\mathbb{ F},\mathbb{G})$ 
defined on a maximal time interval $[0,T_\e^n)$.
 In order to prove that the solution is global, that is $T_\e^n=\infty$ for all $n\in\N,$ we show that the energy functional $\E$ decreases
along $(f_\e^n,g_\e^n)$.
 Indeed, since $\p_x^2 f_\e^n(t),\p_x^2
g_\e^n(t)\in\langle\phi_0,\ldots,\phi_n\rangle$ for all $t\in[0,T_{\e}^n),$  we
may use them as test functions for \eqref{S1}. 
Integrating by parts,  we then find 
 \begin{align}
 \frac{d}{dt}\left(\E(f_\e^n,g_\e^n)\right)=&\int_\cI
((1+R)\p_xf_\e^n+R\p_xg_\e^n)\p_x (\p_tf_\e^n)+R(
\p_xf_\e^n+\p_xg_\e^n)\p_x(\p_tg_\e^n)\, d x\nonumber\\
 =&-\int_\cI ((1+R)\p_x^2f_\e^n+R\p_x^2g_\e^n)\p_tf_\e^n+R(
\p_x^2(f_\e^n+g_\e^n)) \p_tg_\e^n \, d x\nonumber\\
 =&-\mu R^2\int_\cI a_\e^3(g_\e^n)\left|\p_x^3(f_\e^n+g_\e^n)\right|^2dx-\frac{3R^2}{4}\int_\cI a_\e(f_\e^n)a_\e^2(g_\e^n)\left|\p_x^3(f_\e^n+g_\e^n)\right|^2 dx\nonumber\\
&-\int_\cI a_\e(f_\e^n)\left|a_\e(f_\e^n)\p_x^3f_\e^n+\frac{R}{2}(2a_\e(f_\e^n)+3a_\e(g_\e^n))\p_x^3(f_\e^n+g_\e^n)\right|^2 dx\label{E2}
 \end{align}
for all $t\in[0,T_{\e}^n),$
the last equality following similarly  \eqref{E1}.
Particularly, relation \eqref{E2}  ensures the boundedness  of the function
$(\mathbb{F},\mathbb{G})$ on $[0,T_\e^n).$
This means that the Galerkin approximations $(f_\e^n,g_\e^n)$ exist globally in
time for all $n\in\N$. 

\vspace{0.2cm}
\subsection{\bf Convergence of the Galerkin approximations}
$ $ \vspace{0.2cm}

We next identify  an accumulation point of the family $((f_\e^n, g_\e^n))_n$, which is
shown subsequently to be a weak solution of the regularized system in the sense of Theorem \ref{T:1}. 
To this end, let $T\in (0,\infty)$ be an arbitrary constant. 
Invoking \eqref{E2}, we deduce the uniform boundedness\footnote{{All the bounds in this section are uniform in $n\in\N$.}} of 
\begin{align}\label{B1}
& \p_xf_\e^n,\  \p_xg_\e^n \quad \text{ in $L_\infty(0,T; L_2(\cI))$,}\\[1ex]
\label{B2}
& a_\e^{3/2}(g_\e^n)\p_x^3(f_\e^n+g_\e^n), \, a_\e^{1/2}(f_\e^n)a_\e(g^n_\e)\p_x^3(f^n_\e+g^n_\e)\quad \text{in $L_2(Q_T)$,}\\[1ex]
\label{B3}
& a_\e^{1/2}(f_\e^n)\left(a_\e(f_\e^n)\p_x^3f_\e^n+\frac{R}{2}(2a_\e(f_\e^n)+3a_\e(g_\e^n))\p_x^3(f_\e^n+g_\e^n)\right)\quad \text{ in $L_2(Q_T)$,}
 \end{align}
while from \eqref{eq:00} we obtain that
\begin{align}\label{B4}
\int_\cI f_\e^n(t)\,dx=\|f_0\|_{L_1}\quad\text{and}\quad\int_\cI g_\e^n(t)\,dx=\|g_0\|_{L_1}\qquad\text{for all $t\geq0$}.
 \end{align}

First, we observe that Poincar\'e's inequality combined with \eqref{B1}  and \eqref{B4} imply that
 \begin{align}\label{U1}
& f_\e^n,\  g_\e^n \quad \text{are bounded in $L_\infty(0,T; H^1(\cI))$.}
 \end{align}
On the other hand, by construction  we know that  $a_\e\geq\e,$   and we infer  from  the relations   \eqref{B2}-\eqref{B4}, by using Poincar\' e's
inequality again and the uniform boundedness of $(f_\e^n)_n$ and $(g_\e^n)_n$ in $C(\ov Q_T)$, cf. \eqref{U1}, that
 \begin{align}\label{U2}
& f_\e^n,\  g_\e^n \quad \text{ are bounded in $L_2(0,T; H^3(\cI))$.}
 \end{align}

In the next step, we derive uniform bounds for the time derivatives of the Galerkin approximations. 
In order to do so,  we     observe that the relation of \eqref{S1} may be written  in a more concise form as
$\p_tf_\e^n=-\p_x H_f^{\e,n},$
whereby we set
\begin{align*}
H_f^{\e,n}:=&
a^3_\e(f_\e^n)\p_x^3f_\e^n+\frac{R}{2}
\left(2a^3_\e(f_\e^n)+3a^2_\e(f_\e^n)a_\e(g_\e^n)\right)\p_x^3(f_\e^n+g_\e^n)\\[2ex]
=&a^{3/2}_\e(f_\e^n)\left(a_\e^{1/2}(f_\e^n)\left(
a_\e(f_\e^n)\p_x^3f_\e^n+\frac{R}{2}
\left(2a_\e(f_\e^n)+3a_\e(g_\e^n)\right)\p_x^3(f_\e^n+g_\e^n)\right)\right).
\end{align*}
The relations \eqref{B3} and \eqref{U1} yield now that the sequence $(H_f^{\e,n})_n$ is bounded in $L_2(Q_T).$  
Given  $\zeta\in H^1(\cI)$, we define for each $n\in\N$ the truncation  
\[\zeta^n:=\sum_{k=0}^n(\zeta|\phi_k)_{L_2}\phi_k\in\langle\phi_0,\ldots,\phi_n\rangle.\]
Integration by parts then implies that
 \begin{align*}|(\p_tf^n_\e(t)|\zeta)_{L_2}|=&|(\p_tf^n_\e(t)|\zeta^n)_{L_2}|
=|(H_f^{\e,n}(t)|\p_x\zeta_n)_{L_2}|\leq \|H_f^{\e,n}(t)\|_{L_2 }\|\zeta^n\|_{H^1}\\
\leq &\|H_f^{\e,n}(t)\|_{L_2 }\|{\zeta}\|_{H^1}.
\end{align*}
Consequently, for every $t\in[0,T]$, the function  $\p_tf_\e^n(t)$ belongs to  the dual $\left(H^1(\cI)\right)'$ of $H^1(\cI)$  and, integration with respect to time, yields
\[
\|\p_tf_\e^n\|_{L_2(0,T;\left(H^1(\cI)\right)')}\leq \|H_f^{\e,n}\|_{L_2 (Q_T)}.
\]
We claim that a similar estimate is valid also for $\p_t g_\e^n.$ 
Indeed, the second relation of \eqref{S1} may be recast as the equation
$\p_tg_\e^n=-\p_x H_g^{\e,n},$
whereby
 \begin{align}
H_g^{\e,n}:=& \frac{3}{2} a^2_\e(f_\e^n)a_\e(g_\e^n)\p_x^3f_\e^n+\frac{R}{2}\left(2\mu a^3_\e(g_\e^n)+3a^2_\e(f_\e^n)a_\e(g_\e^n)+6a_\e(f_\e^n)a^2_\e(g_\e^n)\right)\p_x^3(f_\e^n+g_\e^n)\nonumber\\
=&\mu R a_\e^{3/2}(g_\e^n)\left(a_\e^{3/2}(g_\e^n)\p_x^3(f_\e^n+g_\e^n)\right)+\frac{3R}{4}a_\e^{1/2}(f_\e^n)a_\e(g_\e^n)\left(a_\e^{1/2}(f_\e^n)
a_\e(g^n_\e)\p_x^3(f^n_\e+g^n_\e)\right)\nonumber\\
&+\frac{3}{2}a_\e^{1/2}(f_\e^n)a_\e(g_\e^n)\left(a_\e^{1/2}(f_\e^n)\left(a_\e(f_\e^n)\p_x^3f_\e^n+\frac{R}{2}\left(2a_\e(f_\e^n)+3a_\e(g_\e^n)\right)\p_x^3(f_\e^n+g_\e^n)\right)\right).\label{HHH}
\end{align}
Gathering   \eqref{B2}, \eqref{B3}, and \eqref{U1}, we see that $(H_g^{\e,n})_n$ is a bounded sequence in $L_2 (Q_T),$
and we  finally conclude  that 
  \begin{align}\label{U3}
 (\p_tf_\e^n)_n,\  (\p_tg_\e^n)_n &\quad \text{are  bounded in $L_2(0,T;  (H^1(\cI))')$.}
\end{align}

Using an argument based on the Aubin-Lions Lemma, cf. Corollary 4 in \cite{Si87}, together with the continuity of the embeddings 
\[
\begin{array}{lll}
&H^1(\cI)\overset{comp}\hookrightarrow C^\alpha(\ov\cI)\hookrightarrow (H^1(\cI))' \\[1ex]
 &H^3(\cI)\overset{comp}\hookrightarrow C^{2+\alpha}(\ov\cI)\hookrightarrow (H^1(\cI))'
\end{array} \qquad \text{for $\alpha\in[0,1/2),$}
\]
we conclude from \eqref{U2} and \eqref{U3}
that the sequences 
$(f_\e^n)_n$ and $(g_\e^n)_n$ are both  relatively compact in $C([0,T],C^\alpha(\ov\cI))\cap L_2(0,T;C^{2+\alpha}(\ov\cI))$
for all $\alpha\in[0,1/2).$
Hence, using a diagonal procedure, we find  functions 
$f_\e$ and $g_\e$ and subsequences of  $(f_\e^n)$ and $(g_\e^n)$  (not relabeled)
 such that
 \begin{equation}\label{U4}
f_\e^n\to f_\e\quad\text{and}\quad  g_\e^n\to g_\e\quad \text{in $C([0,T],C^\alpha(\ov\cI))\cap L_2(0,T;C^{2+\alpha}(\ov\cI))$}
\end{equation}
for all $\alpha\in[0,1/2).$
Furthermore,  let us observe that the relations \eqref{U1}, \eqref{U2}, and \eqref{U4} ensure that the
limit functions belong also to $f_\e,g_\e \in L_\infty(0,T;H^1(\cI))\cap L_2(0,T;H^3(\cI))$ and that  
 \begin{equation}\label{U5}
\p_x^k f_{\e}^n\rightharpoonup  \p_x^kf_\e\quad\text{and}\quad  \p_x^k
g_{\e}^n\rightharpoonup  \p_x^kg_\e\quad \text{ in $L_2(Q_T)$ for $k=1,2,3.$} 
\end{equation}
Finally,  from \eqref{U3} we obtain that  $\p_tf_\e,\  \p_tg_\e \in L_2(0,T; (H^1(\cI))')$, and 
 \begin{align}\label{U6}
 \p_tf_\e^n\rightharpoonup\p_tf_\e,\ 
\p_tg_\e^n\rightharpoonup\p_tg_\e\qquad\text{in $ L_2(0,T;  (H^1(\cI))')$}.
\end{align}
The fact that $(f_\e, g_\e)$ can be defined globally follows
by using a standard Cantor  diagonal argument (choosing  a sequence $T_n\nearrow \infty$).

\vspace{0.2cm}
\subsection{Construction of the weak solutions for the regularized system}
$ $ \vspace{0.2cm}

In this last part of Section \ref{S:2} we prove that the functions $(f_\e,g_\e)$ constructed in (\ref{U4}) are weak solutions of \eqref{S1}, \eqref{eq:bc1}, and \eqref{eq:bc2}, and enjoy all the properties
stated in Theorem \ref{T:1}. 
First, let us observe that the functions $f_\e$ and $ g_\e$ possess the regularity and integrability properties required in Theorem \ref{T:1}. 
Moreover, because  $f_0\in H^1(\cI) $ and due to \eqref{U4}, we have that
$f_\e(0)=f_0$ and $g_\e(0)=g_0$.
Furthermore, combining  \eqref{B4} and \eqref{U4}, it follows that the identities \eqref{12} are satisfied.

Concerning \eqref{bb3}, we note that  \eqref{U4} guarantees that $f_\e^n(T)\to f_\e(T)$ in $C^2(\ov \cI)$ for almost all $T\geq0.$ 
Because $\p_xf_\e^n(T)=0$ at $x=0,L,$ the desired claim \eqref{bb3} for $f_\e$  (and similarly for $g_\e$) is immediate. 

We next prove that the  energy estimate \eqref{13} is satisfied by the functions $(f_\e,g_\e)$. 
To this end, we infer from the relations \eqref{B2}, \eqref{B3}, \eqref{U4}, and \eqref{U5}, by using also the Lipschitz
continuity of the map $a_\e$ and  after extracting further subsequences of $(f_\e^n)$ and $(g_\e^n)$ (not relabeled) that we have the following weak convergences in $L_2(Q_T):$
\begin{align*}
 &a_\e^{3/2}(g_\e^n)\p_x^3(f_\e^n+g_\e^n)\rightharpoonup a_\e^{3/2}(g_\e)\p_x^3(f_\e+g_\e),\\
&a_\e^{1/2}(f_\e^n)a_\e(g^n_\e)\p_x^3(f^n_\e+g^n_\e)\rightharpoonup a_\e^{1/2}(f_\e)a_\e(g_\e)\p_x^3(f_\e+g_\e),\\
 &a_\e^{1/2}(f_\e^n)\left(a_\e(f_\e^n)\p_x^3f_\e^n+\frac{R}{2}(2a_\e(f_\e^n)+3a_\e(g_\e^n))\p_x^3(f_\e^n+g_\e^n)\right)\\
&\hspace{4cm}\rightharpoonup a_\e^{1/2}(f_\e)\left(a_\e(f_\e)\p_x^3f_\e+\frac{R}{2}(2a_\e(f_\e)+3a_\e(g_\e))\p_x^3(f_\e+g_\e)\right).
\end{align*}
Recalling that $f_\e^n(t)\to f_\e(t)$, $g_\e^n(t)\to g_\e(t)$ in $C^2(\ov \cI)$ for almost all $t\geq0,$
and that the initial data  $f_0,g_0$ belong to $ H^1(\cI),$ we conclude after integrating \eqref{E2} with
respect to time and passing  to $\liminf_{n\to\infty} $  that
the desired energy inequality \eqref{13} is satisfied.

To finish the proof of Theorem \ref{T:1}, we are only left to  prove  relations \eqref{11}.
Let therefore $\xi\in L_2(0,T;H^1(\cI))$ be given,  and define for each $n\in\N$ the truncation
\[
\xi^n(t,\cdot):=\sum_{k=0}^n(\xi(t,\cdot)|\phi_k)_{L_2}\phi_k, \qquad t\in(0,T).
\] 
Using integrating by parts, we find in a similar way as before that
\begin{align*}
 \int_\cI\p_t f_\e^n(t) \xi^n(t) \, dx=-\int_\cI\xi^n(t)\p_x H_f^{\e, n}(t)\, dx=\int_\cI H_f^{\e, n}(t)\p_x \xi^n(t)\, dx,
\end{align*}
whence, we have 
\begin{align}\label{L1}
 \int_0^T\langle\p_t f_\e^n(t)| \xi^n(t)\rangle \, dt=&\int_0^T(\p_t f_\e^n(t)|
\xi^n(t))_{L_2} \, dt=\int_{Q_T} H_f^{\e, n}(t)\p_x \xi^n(t)\, dx
\end{align}
for all $n\geq0.$
Since by Lebesgue's dominated convergence  $\xi^n\to\xi$ in 
$L_2(0,T;H^1(\cI))$ we find, together with \eqref{U6}, that
 \begin{align}\label{L2}
\int_0^T\langle\p_t f_\e^n(t)| \xi^n(t)\rangle \, dt\to\int_0^T\langle\p_t f_\e(t) | \xi(t)\rangle \, dt.
\end{align}
Furthermore, \eqref{B2}, \eqref{U1}, \eqref{U4}, and \eqref{U5} ensure that, after extracting further subsequences,
we have
 $H_f^{\e,n}\rightharpoonup H_f^{\e}$ in  $L_2(Q_T)$, whereby we set
\begin{align}\label{L3}
 & H_f^{\e}:=a^3_\e(f_\e )\p_x^3f_\e  
+\frac{R}{2}\left(2a^3_\e(f_\e )+3a^2_\e(f_\e )a_\e(g_\e )\right)\p_x^3(f_\e
+g_\e ).
\end{align}
Letting now $n\to\infty$ in \eqref{L1}, we obtain from \eqref{L2} and  \eqref{L3} the first relation  of \eqref{11}. 
On the other hand,  combining the relations \eqref{B2}, \eqref{B3}, \eqref{U1},
\eqref{U4}, and \eqref{U5} we may assume that $H_g^{\e,n}\rightharpoonup H_g^\e$
in  $L_2(Q_T)$, with $H_g^\e$  given by
\begin{align*}
 &  H_g^\e:=\frac{3}{2}  a^2_\e(f_\e)a_\e(g_\e)\p_x^3f_\e+\frac{R}{2}\left(2\mu a^3_\e(g_\e)+3a^2_\e(f_\e)a_\e(g_\e)+6a_\e(f_\e)a^2_\e(g_\e)\right)\p_x^3(f_\e+g_\e).
\end{align*}
Repeating the arguments presented above we conclude that the second identity of \eqref{11} is also satisfied, and the proof of Theorem \ref{T:1} is complete.

Let us remark that we do not know  whether the weak  solutions $(f_\e,g_\e)$ found in
Theorem \ref{T:1} are non-negative.
The next lemma though,  together with the convergence results that we will provide in
the next section yields the non-negativity of the weak solutions of problem
\eqref{P}, which are found as being accumulation points of the family $((f_\e,g_\e))_{\e\in(0,1]}$, cf. Lemma \ref{L:1} and Corollary \ref{C:1} below.   
To this end, we introduce the following notation.
Pick a function $\varphi\in C^\infty(\R)$, which is non-negative, has support
contained in $[-1,0]$, and satisfies
 \[
\int_{\R}\varphi(x)\,dx=1.
\]
Moreover, let the function $\chi_1:\R\to\R$ be defined  by the relation
\[
\chi_1(x):=-\int_0^x\int_s^\infty\varphi(\tau)\,d\tau ds\qquad\text{for $x\in\R$,}
\]
and $(\chi_\delta)_{\delta>0}$ be the associated mollifier, that is
$\chi_\delta(x):=\delta\chi_1(x/\delta)$ for $x\in\R$ and $\delta>0.$
The following properties of $(\chi_\delta)_{\delta>0}$ play an important role in
the proof of Lemma \ref{L:1}:
\begin{align}
&\|\chi_\delta-\max\{-\id,0\}\|_{L_\infty(\R)}\leq \delta,\label{M1}\\
&\text{$\|\chi_\delta'\|_{L_\infty(\R)}\leq 1$, \,
$\|\chi_\delta''\|_{L_\infty(\R)}\leq \delta^{-1}\|\varphi\|_{L_\infty(\R)}$, \,  and 
$\|\chi_\delta'''\|_{L_\infty(\R)}\leq \delta^{-2}\|\varphi'\|_{L_\infty(\R)}$}\label{M2}
\end{align}
for all $\delta>0.$

\begin{lemma}\label{L:1} 
The functions $(f_\e,g_\e)$ found in Theorem \ref{T:1} satisfy
\begin{equation}
 \left|\int_\cI\chi_{\sqrt{\e}}(f_\e(T))\, dx\right|\leq  C\sqrt{T}\e\text{\quad and\quad}\left|\int_\cI\chi_{\sqrt{\e}}(g_\e(T))\, dx\right|\leq
C\sqrt{T\e}\label{Neg}
\end{equation}
for all $\e\in(0,1] $ and all $T\geq0.$
 \end{lemma}

Before proving Lemma \ref{L:1}  let us draw the conclusion that all
accumulation points of  the family $((f_\e,g_\e))_{\e\in(0,1]}$ in $C(\ov Q_T,\R^2),$ with $T>0$,
are non-negative functions.
\begin{cor}\label{C:1}
 Assume that  there exists a  sequence $(\e_k)_{k}\subset(0,1]$ with $\e_k\searrow0$
and a pair  $(f,g)\in C(\ov Q_T,\R^2)$ such that 
\begin{equation}\label{Ass}
(f_{\e_k},g_{\e_k})\to (f,g)\quad\text{in $C(\ov Q_T,\R^2)$.}
\end{equation} 
Then, $f$ and $g$ are both non-negative functions in $Q_T$.
\end{cor}
\begin{proof} In virtue of \eqref{M2}, we have 
 \begin{align*}
&\|\chi_{\sqrt{\e_k}}(f_{\e_k})-\max\{-f,0\}\|_{L_\infty(Q_T)}\\
&\leq\|\chi_{\sqrt{{\e_k}}}(f_{\e_k})-\chi_{\sqrt{\e_k}}(f)\|_{L_\infty(Q_T)}+\|\chi_{\sqrt{\e_k}}(f)-\max\{-f,0\}\|_{L_\infty(Q_T)}\\
&\leq\| f_{\e_k}-f\|_{L_\infty(Q_T)}+\|\chi_{\sqrt{\e_k}}(f)-\max\{-f,0\}\|_{L_\infty(Q_T)}.
\end{align*}
Whence, our assumption \eqref{Ass} guarantees the convergence 
$\chi_{\sqrt{\e_k}}(f_{\e_k})\to\max\{-f,0\}$ in $C(\ov Q_T).$ 
Letting now $k\to\infty$ in the first inequality of \eqref{Neg} yields
\[
\int_\cI\max\{-f(t),0\}\, dx=0
\]
for all $t\in[0,T].$ This is the desired assertion for $f$. 
The proof of the non-negativity   of $g$ follows similarly.
\end{proof}

\begin{proof}[Proof of Lemma \ref{L:1}] 
Let $\delta>0$ be given. 
Since $\chi_\delta'(f_\e^n(t))\in H^1(\cI) $, we compute that
\begin{align*}
 \frac{d}{dt}\int_\cI\chi_\delta(f_\e^n(t))\, dx=&\int_\cI\chi_\delta'(f_\e^n(t))\p_tf_\e^n(t)\, dx=\int_\cI\p_tf_\e^n(t)\sum_{k=0}^n(\chi_\delta'(f_\e^n(t))|\phi_k)_{L_2}\phi_k\, dx\\
=&\int_\cI H^{\e,n}_f(t)\sum_{k=0}^n\p_x\left((\chi_\delta'(f_\e^n(t))|\phi_k)_{L_2}\phi_k\right)\,dx,
\end{align*}
relation which is satisfied for all $t\geq0.$
The assertions \eqref{Neg} are obviously true when $T=0,$ so let us assume that $T>0.$
Integration the previous identities with respect to time on $[0,T]$ shows that  
\begin{align}
\int_\cI\chi_\delta(f_\e^n(T))\, dx=\int_\cI\chi_\delta(f_\e^n(0))\, dx+\int_{Q_T}H^{\e,n}_f\sum_{k=0}^n (\chi_\delta'(f_\e^n)|\phi_k)_{L_2}\p_x\phi_k \, dxdt.\label{gaucho}
\end{align}
In order to let $n\to\infty$ in \eqref{gaucho}, we first observe   
\begin{equation}\label{eq}
\sum_{k=0}^n(\chi_\delta'(f_\e^n )|\phi_k)_{L_2}\p_x\phi_k\to \chi_\delta''(f_\e)\p_xf_\e \qquad\text{in $L_2(Q_T).$}
\end{equation}
Indeed, we have
\begin{align*}
\chi_\delta''(f_\e)\p_xf_\e-\sum_{k=0}^n(\chi_\delta'(f_\e^n )|\phi_k)_{L_2}\p_x\phi_k =&\left(\chi_\delta''(f_\e)\p_xf_\e-\sum_{k=0}^n(\chi_\delta'(f_\e)|\phi_k)_{L_2}\p_x\phi_k\right)\\
&+\sum_{k=0}^n\left(\chi_\delta'(f_\e)-\chi_\delta'(f_\e^n)|\phi_k\right)_{L_2}\p_x\phi_k,
\end{align*}
and the convergence of the first term to zero follows by using Lebesgue's
dominated convergence theorem together with the fact that $\chi_\delta'(f_\e(t))\in H^1(\cI)$ for all $t\geq0$.
On the other hand, the  reminding sum is the  truncation of the Fourier series of 
$\chi_\delta''(f_\e^n)\p_xf_\e^n-\chi_\delta''(f_\e)\p_xf_\e$ and, using 
\eqref{M2}, may be estimated as follows 
\begin{align*}
 \left\|\sum_{k=0}^n(\chi_\delta'(f_\e^n)-\chi_\delta'(f_\e)|\phi_k)\p_x\phi_k\right\|^2_{L_2(Q_T)}
\leq&\|\chi_\delta''(f_\e^n)\p_xf_\e^n-\chi_\delta''(f_\e)\p_xf_\e\|^2_{L_2(Q_T)}\\
 \leq&2\|\chi_\delta''(f_\e^n)-\chi_\delta''(f_\e)\|_{L_\infty(Q_T)}^2\|\p_xf_\e^n\|^2_{L_2(Q_T)}\\
&+2\|\chi_\delta''(f_\e^n)\|_{L_\infty(Q_T)}
^2\|\p_xf_\e^n-\p_xf_\e\|^2_{L_2(Q_T)}\\
 \leq&2\delta^{-4}\|\varphi'\|_{L_\infty(\R)}^2\|f_\e^n-f_\e\|_{L_\infty(Q_T)}^2\|\p_xf_\e^n\|^2_{L_2(Q_T)}\\
&+2\delta^{-2}\|\varphi\|_{L_\infty(\R)}^2\|\p_xf_\e^n-\p_xf_\e\|^2_{L_2(Q_T)},
\end{align*}
the desired estimate \eqref{eq} being now a consequence of \eqref{U1} and \eqref{U4}.

Thus, letting $n\to\infty$ in \eqref{gaucho} and taking into account that
$f_\e(0)=f_0\geq0,$ we obtain the following identity for the weak solution of
\eqref{S1} found in  Theorem \ref{T:1}
\begin{align*}
 &\int_\cI\chi_\delta(f_\e(T))\, dx=\int_{Q_T}H_f^\e\chi_\delta''(f_\e )\p_xf_\e\, dxdt.
\end{align*}
Since $\chi_\delta''=0$ on $\R\setminus(-\delta,0),$ H\"older's inequality leads
us to 
\begin{align*}
 \left(\int_\cI\chi_\delta(f_\e(T))\, dx\right)^2\leq&\left(\int_{[-\delta\leq f_\e\leq 0]}\left| H_f^\e\right|\chi_\delta''(f_\e )|\p_xf_\e|\, dxdt\right)^2\\
\leq&\int_{[-\delta\leq f_\e\leq 0]}a_\e(f_\e)\left|a_\e(f_\e)\p_x^3f_\e+\frac{R}{2}(2a_\e(f_\e)+3a_\e(g_\e))\p_x^3(f_\e+g_\e)\right|^2 dxdt\\
&\times\int_{[-\delta\leq f_\e\leq 0]} a^3_\e(f_\e)\chi_\delta''(f_\e )^2|\p_xf_\e|^2\, dxdt.
\end{align*}
We choose now $\delta:=\sqrt{\e}.$ 
Recalling that $a_\e\equiv\e$ on $(-\infty,0]$,  the energy inequality \eqref{13} together with   \eqref{M2}  imply that 
\begin{align*}
 \left|\int_\cI\chi_{\sqrt{\e}}(f_\e(T))\, dx\right|&\leq C\left(\int_{[-\sqrt{\e}\leq f_\e\leq 0]}  \e^3\chi_{\sqrt{\e}}''(f_\e)^2|\p_xf_\e|^2\, dxdt\right)^{1/2}\\
&\leq C\e\|\varphi\|_{L_\infty(\R)}\left(\int_{Q_T}|\p_xf_\e|^2\,dxdt\right)^{1/2}\leq C\sqrt{T}\e,
\end{align*}
which is the desired estimate \eqref{Neg} for $f_\e.$ 
Concerning the second estimate of \eqref{Neg}, similar arguments to those presented above yield that 
\begin{align*}
 &\int_\cI\chi_\delta(g_\e(T))\, dx=\int_{Q_T} H_g^\e \chi_\delta''(g_\e )\p_xg_\e\, dxdt
\end{align*}
for all $T>0$ and $\delta>0$. 
Writing $H_g^\e$  as the sum of three terms, cf. \eqref{HHH}, we obtain from
H\"older's inequality and the estimate \eqref{13} the following inequalities
\begin{align*}
 &\left|\int_\cI\chi_\delta(g_\e(T))\, dx\right|\\
&\leq\mu R\int_{[-\delta\leq g_\e\leq0]} a^{3/2}_\e(g_\e)\chi_\delta''(g_\e)|\p_xg_\e|\left|a^{3/2}_\e(g_\e)\p_x^3(f_\e+g_\e)\right| \, dxdt\\
&\phantom{\leq\,}+\frac{3R}{4}\int_{[-\delta\leq g_\e\leq0]}a_\e^{1/2}(f_\e^n)a_\e(g_\e^n) \chi_\delta''(g_\e)|\p_xg_\e|\left|a_\e^{1/2}(f_\e^n)a_\e(g^n_\e)\p_x^3(f^n_\e+g^n_\e)\right|\,dxdt\\
&\phantom{\leq\,}+\frac{3}{2}\int_{[-\delta\leq g_\e\leq0]} a_\e^{1/2}(f_\e^n)a_\e(g_\e^n) \chi_\delta''(g_\e )|\p_xg_\e|\\
&\phantom{\leq\,}\hspace{2.5cm}\times\left|a_\e^{1/2}(f_\e^n)\left(a_\e(f_\e^n)\p_x^3f_\e^n+\frac{R}{2}\left(2a_\e(f_\e^n)+3a_\e(g_\e^n)\right)\p_x^3(f_\e^n+g_\e^n)\right)\right|\,dxdt\\
&\leq C\left(\int_{[-\delta\leq g_\e\leq0]}a^2_\e(g_\e)(a_\e(f_\e)+a_\e(g_\e))\chi_\delta''(g_\e )^2|\p_xg_\e|^2\, dxdt\right)^{1/2},
\end{align*}
and, when $\delta=\sqrt{\e},$ we arrive at the following estimate 
\begin{align*}
 \left|\int_\cI\chi_{\sqrt{\e}}(g_\e(T))\, dx\right|&\leq C \sqrt{\e}\|\varphi\|_{L_\infty(\R)}\left(\int_{Q_T}|\p_xg_\e|^2 \, dxdt\right)^{1/2}\leq C\sqrt{T\e}.
\end{align*}
This proves the lemma.
 \end{proof}

\section{Existence of weak solutions for  the original problem}\label{S:3}

This last section is devoted to the proof of our main result Theorem \ref{T:M}.
Therefore, we collect first some estimates for the family of weak solutions $((f_\e,g_\e))_{\e\in(0,1]} $ of the approximating problems \eqref{S1}, \eqref{eq:bc1}, and \eqref{eq:bc2}. 
Considering now    $\e\in(0,1]$ as a parameter,   we deduce from  \eqref{11}, \eqref{12}, and \eqref{13} the uniform boundedness  of 
 \begin{align}\label{Fi1}
& \p_xf_\e,\  \p_xg_\e \quad \text{  in $L_\infty(0,T;L_2(\cI))$,}\\
\label{Fi2}
& \p_tf_\e,\  \p_tg_\e\quad \text{ in $L_2(0,T;\left(H^1(\cI)\right)')$,}\\
\label{Fi3}
& a_\e^{3/2}(g_\e)\p_x^3(f_\e+g_\e),\, a_\e^{1/2}(f_\e)a_\e(g_\e)\p_x^3(f_\e+g_\e)\quad \text{in $L_2(Q_T)$,}\\
\label{Fi4}
&a_\e^{1/2}(f_\e)\left(a_\e(f_\e)\p_x^3f_\e+\frac{R}{2}(2a_\e(f_\e)+3a_\e(g_\e))\p_x^3(f_\e+g_\e)\right)\quad \text{in $L_2(Q_T)$}
 \end{align}
for all $T>0$.
Recalling also \eqref{12}, the   arguments used in the previous section  ensure the existence of a sequence $(\e_k)_{k}\subset(0,1]$ with $\e_k\searrow0$ and functions 
\[
f,g\in L_2(0,T;H^1(\cI))\cap \left(\cap_{\alpha\in (0,1/2] }C([0,T], C^\alpha(\ov\cI))\right),
\]
having the property that 
\begin{align}\label{F1}
 &f_{\e_k}\to f,\quad g_{\e_k} \to g\qquad\text{in $C([0,T], C^\alpha(\ov\cI))$ for all $\alpha\in[0,1/2);$}\\
&f_{\e_k}\rightharpoonup f,\quad g_{\e_k} \rightharpoonup g\qquad\text{in $L_2(0,T;H^1(\cI)).$}\label{F2}\\
\end{align}
Particularly, \eqref{F1} implies that for almost  every $t\in[0,T]$ we have
\begin{equation}\label{point}
\text{ $\p_xf_{\e_k}(t)\rightharpoonup \p_xf(t)$, \quad  $\p_xg_{\e_k}(t) \rightharpoonup \p_xg(t)$  \qquad in $L_2(\cI)$,}                        
\end{equation}
and it follows now directly from  \eqref{13}  that $f,g\in L_\infty(0,T; H^1(\cI)).$ 
Moreover, the convergence \eqref{F1} together with the   Corollary \ref{C:1} yield the non-negativity of the limits $f$ and $g.$
The latter property combined with the relation \eqref{12} ensure the desired mass conservation property claimed by Theorem \ref{T:M} $(c)$.
Let us also observe that the claim $(b)$ of
Theorem \ref{T:M} is a simple consequence of the convergence \eqref{F1} and of the relations $f_\e(0)=f_0$ and $g_\e(0)=0$ for all $\e\in(0,1]$. 

We next establish the identities $(d)$ of Theorem \ref{T:M}.   
For this  let  $\xi\in C^\infty(\ov Q_T)$ be given and,  for every
$\e\in(0,1],$ let $((f_\e^n,g_\e^n))_n$ be the sequence found in Section
\ref{S:2} to converge  towards the weak
solution $(f_\e,g_\e)$ of problem \eqref{S1}.
Integrating by parts, we then find that
\begin{align}
 \int_0^T\int_\cI\p_tf_\e^n \xi \, dxdt={ \int_\cI f_\e^n(T,x)  \xi(T,x )\, dx}-\int_\cI f_\e^n(0)  \xi(0,x )\, dx-\int_\cI\int_0^Tf_\e^n \p_t\xi \,dxdt.\label{parti}
\end{align}
Hence, letting $n\to\infty$ in \eqref{parti}, we deduce in virtue of \eqref{U4},
\eqref{U6}, and of the first identity in \eqref{11} that
\begin{align}
\int_{Q_T}H_f^\e\p_x\xi\, dx dt={ \int_\cI f_\e(T,x)  \xi(T,x )\, dx}-\int_\cI f_0(x)  \xi(0,x )\, dx-\int_{Q_T}f_\e \p_t\xi \,dxdt.\label{limita}
\end{align}
Particularly, recalling \eqref{F1},  it is easy to see
that along the subsequence $((f_{\e_k},g_{\e_k}))_k,$ the right hand side of the equality \eqref{limita} converges
towards the corresponding  quantities appearing in the equation \eqref{I1}.

In order to study the behaviour of the left hand side  of the relation \eqref{limita}, we
define for every $m\in\N$, $m\geq1,$ the open subsets
\begin{align*}
\cP_f^m&:=\{(t,x)\in(0,T)\times\cI\,:\, \text{$f(t,x)>1/m$}\},\\
\cP_g^m&:=\{(t,x)\in(0,T)\times\cI\,:\, \text{$g(t,x)>1/m$}\},
\end{align*}
of $Q_T$, and observe that $\cP_f=\cup_{m=1}^\infty \cP_f^m$ and $\cP_g=\cup_{m=1}^\infty\cP_g^m.$
Let now $m\geq1$ be fixed.
Due to \eqref{F1}, we may find a positive integer $k_0$ with the property that
$f_{\e_k}(t,x)>(2m)^{-1}$ and $g_{\e_k}(t,x)>(2m)^{-1}$ for all $(t,x)\in \cP_f^m\cap \cP_g^m$ and all $k\geq k_0.$
Thanks to \eqref{Fi3} and \eqref{Fi4},   the sequences $(\p_x^3 f_{\e_k})_k$ and $(\p_x ^3g_{\e_k})_k$ are both bounded in $L_2(\cP_f^m\cap \cP_g^m),$
and, up to the extraction of a diagonal subsequence, we may assume that
\begin{equation}\label{help}
\p_x^3f_{\e_k}\rightharpoonup \p_x^3f,\quad \p_x^3g_{\e_k} \rightharpoonup\p_x^3g\qquad\text{in $L_2(\cP_f^m\cap \cP_g^m) $ }
\end{equation}
for all $m\geq1.$
Because $(H^{\e_k}_f)_k$  and $(a_{\e_k}^{-3/2}(f_{\e_k})H^{\e_k}_f)_k$ are bounded in $L_2(Q_T), $ cf. \eqref{12}, \eqref{Fi1},  and \eqref{Fi4},
we can also presuppose that there exist  functions  $ H_f,  j_f\in L_2(Q_T)$ such that  
\begin{equation}\label{help1}
 H^{\e_k}_f\rightharpoonup  H_f,\quad \left(a_{\e_k}(f_{\e_k})\right)^{-3/2}
H^{\e_k}_f\rightharpoonup  j_f\qquad\text{in $L_2(Q_T)$}.
\end{equation}
Using the convergences \eqref{help} and \eqref{F1}, we may identify the weak limits in
\eqref{help1} in the set where  $f$ and $g$ are both positive
\begin{equation*}
\begin{array}{llll}
& H_f=f^3\p_x^3f+\frac{R}{2}\left(2f^3+3f^2g\right)\p_x^3(f+g)\\[1ex]
& j_f=f^{1/2}\left(f\p_x^3f+\frac{R}{2} (2f+3fg)\p_x^3(f+g)\right)
\end{array}
\qquad\text{in $L_2(\cP_f\cap \cP_g)$.}
\end{equation*}
Moreover, because of $|a_{\e_k}(f_{\e_k})-f|\leq \e_k+|f_{\e_k}-f|$ for all $k\geq0,$ we may identify $ H_f$ in the large set $\cP_g.$  
Indeed, by the dominated convergence theorem   $a_{\e_k}^{3/2}(f_{\e_k})\to f^{3/2}$ in $L_2(Q_T),$  which shows,
together with  \eqref{help1}, that
$ H_f=f^{3/2} j_f$ in $L_2(Q_T).$
Summarizing, we have shown that
\begin{equation}\label{EQ}
 H_f=\left(f^3\p_x^3f+\frac{R}{2}\left(2f^3+3f^2g\right)\p_x^3(f+g)\right)\mathbf{1}_{(0,\infty)}(f)\qquad\text{in $L_2(\cP_g)$,}
\end{equation}
and the desired assertion \eqref{I1} follows now  at once.
The   identity \eqref{I2} is obtained by using similar arguments.
Indeed, in this case it is possible to identify first the weak limit $ H_g$ of (a subsequence of) $(H_g^{\e_k})_k$  in $L_2(Q_T).$
On the other hand, because of \eqref{Fi3},  there exist  functions $j_g ,  j_{f,g}\in L_2(Q_T)$ such that
\begin{equation}\label{help2}
 a_{\e_k}^{3/2}(g_{\e_k})\p_x^3(f_{\e_k}+g_{\e_k})\rightharpoonup j_{g},\quad
a_{\e_k}^{1/2}(f_{\e_k})a_{\e_k}(g_{\e_k})\p_x^3(f_{\e_k}+g_{\e_k})\rightharpoonup   j_{f,g}\qquad\text{in $L_2(Q_T)$}.
\end{equation}
Again, due to \eqref{F1} and \eqref{help}, we identify $ j_{g}=g^{3/2}\p_x^3(f+g)$ and $ j_{f,g}=f^{1/2}g\p_x^3(f+g)$ in
$L_2(\cP_f\cap \cP_g)$.
Writing  $H_g^{\e_k}$ in a similar manner as in \eqref{HHH},  the dominated convergence theorem shows then   
\begin{equation*}
 H_g=\mu R g^{3/2} j_g +\frac{3R}{4}f^{1/2}g j_{f,g}+\frac{3}{2}f^{1/2}g j_f\qquad\text{in $L_2(Q_T)$,}
\end{equation*}
and therefore
\[
 H_g=\left(\frac{3}{2} f^2g\p_x^3f+\frac{R}{2}\left(2\mu g^3+3f^2g+6fg^2\right)\p_x^3(f+g)\right)\mathbf{1}_{(0,\infty)}(g)\qquad\text{in $L_2(\cP_f)$.}
\]
The assertion \eqref{I2} is now immediate. 

Finally,  we collect \eqref{point}, \eqref{help1}, and \eqref{help2}, and 
pass to $\liminf_{k\to\infty}$ in the energy inequality \eqref{13} to obtain    the desired claim $(e)$ of Theorem \ref{T:M}.

Because $T$ was chosen arbitrary, we may again pick a sequence $T_n\nearrow \infty$ and, extracting a diagonal sequence of $((f_{\e_k}, g_{\e_k}))_{k}$ 
we may assume that $f$ and $g$ are globally defined and the claims of Theorem \ref{T:M}
are true for all $T>0.$

\bibliographystyle{abbrv}
\bibliography{Lit}

\begin{thebibliography}{10}

\bibitem{B93}
F.~Bernis.
\newblock {Viscous flows, fourth order nonlinear degenerate parabolic equations
  and singular elliptic problems}.
\newblock In {\em {Free boundary problems: theory and applications ({T}oledo,
  1993)}}, volume 323 of {\em {Pitman Res. Notes Math. Ser.}}, pages 40--56.
  Longman Sci. Tech., Harlow, 1995.

\bibitem{BF90}
F.~Bernis and A.~Friedman.
\newblock {Higher order nonlinear degenerate parabolic equations}.
\newblock {\em J. Differential Equations}, 83(1):179--206, 1990.

\bibitem{BP94}
A.~L. Bertozzi and M.~Pugh.
\newblock {The lubrication approximation for thin viscous films: the moving
  contact line with a ``porous media'' cut-off of van der {W}aals
  interactions}.
\newblock {\em Nonlinearity}, 7(6):1535--1564, 1994.

\bibitem{BP96}
A.~L. Bertozzi and M.~Pugh.
\newblock {The lubrication approximation for thin viscous films: regularity and
  long-time behavior of weak solutions}.
\newblock {\em Comm. Pure Appl. Math.}, 49(2):85--123, 1996.

\bibitem{BG98}
M.~Bertsch, R.~{Dal Passo}, H.~Garcke, and G.~Gr{\"u}n.
\newblock {The thin viscous flow equation in higher space dimensions}.
\newblock {\em Adv. Differential Equations}, 3(3):417--440, 1998.

\bibitem{BLxx}
A.~Blanchet and {\relax Ph}.~{Lauren\c cot}.
\newblock {The parabolic-parabolic {K}eller-{S}egel system with critical
  diffusion as a gradient flow in $\mathbb{R}^d$, $d \ge 3$}.
\newblock preprint arXiv:1203.3573.

\bibitem{EHLW12}
J.~Escher, M.~Hillairet, {\relax Ph}.~{Lauren\c cot}, and C.~Walker.
\newblock {Weak solutions to a thin film model with capillary effects and
  insoluble surfactant}.
\newblock {\em Nonlinearity}, 25:2423--2441, 2012.

\bibitem{ELM11}
J.~Escher, {\relax Ph}.~{Lauren\c cot}, and B.-V. Matioc.
\newblock {Existence and stability of weak solutions for a degenerate parabolic
  system modelling two-phase flows in porous media}.
\newblock {\em Ann. Inst. H. Poincar{\'e} Anal. Non Lin{\'e}aire},
  28(4):583--598, 2011.

\bibitem{EMM12}
J.~Escher, A.-V. Matioc, and B.-V. Matioc.
\newblock {Modelling and analysis of the {M}uskat problem for thin fluid
  layers}.
\newblock {\em J. Math. Fluid Mech.}, 14:267--277, 2012.

\bibitem{EMM12b}
J.~Escher, A.-V. Matioc, and B.-V. Matioc.
\newblock {Thin-film approximations of the two-phase {S}tokes problem}.
\newblock {\em Nonlinear Anal.}, 2012.
\newblock DOI:10.1016/j.na.2012.07.034.

\bibitem{EM12x}
J.~Escher and B.-V. Matioc.
\newblock {Existence and stability of solutions for a strongly coupled system
  modelling thin fluid films}.
\newblock {\em NoDEA Nonlinear Differential Equations Appl.}, 2012.
\newblock DOI 10.1007/s00030-012-0166-1.

\bibitem{GW06}
H.~Garcke and S.~Wieland.
\newblock {Surfactant spreading on thin viscous films: nonnegative solutions of
  a coupled degenerate system}.
\newblock {\em SIAM J. Math. Anal.}, 37(6):2025--2048 (electronic), 2006.

\bibitem{GO03}
L.~Giacomelli and F.~Otto.
\newblock {Rigorous lubrication approximation}.
\newblock {\em Interfaces Free Bound.}, 5(4):483--529, 2003.

\bibitem{GP08}
M.~G{\"u}nther and G.~Prokert.
\newblock {A justification for the thin film approximation of {S}tokes flow
  with surface tension}.
\newblock {\em J. Differential Equations}, 245(10):2802--2845, 2008.

\bibitem{LM12x}
{\relax Ph}.~{Lauren\c cot} and B.-V. Matioc.
\newblock {A gradient flow approach to a thin film approximation of the
  {M}uskat problem}.
\newblock {\em Calc. Var. Partial Differential Equations}, 2012.
\newblock DOI 10.1007/s00526-012-0520-5.

\bibitem{LM12xx}
{\relax Ph}.~{Lauren\c cot} and B.-V. Matioc.
\newblock {A thin film approximation of the Muskat problem with gravity and
  capillary forces}.
\newblock 2012.
\newblock preprint.

\bibitem{BM12}
B.-V. Matioc.
\newblock {Non-negative global weak solutions for a degenerate parabolic system
  modeling thin films driven by capillarity}.
\newblock {\em Proc. Roy. Soc. Edinburgh Sect. A}, 142 (5):1071--1085, 2012.

\bibitem{MP12}
B.-V. Matioc and G.~Prokert.
\newblock {{H}ele-{S}haw flow in thin threads: {A} rigorous limit result}.
\newblock {\em Interfaces Free Bound.}, 14(2):205--230, 2012.

\bibitem{Si87}
J.~Simon.
\newblock {Compact sets in the space {$L^p(0,T;B)$}}.
\newblock {\em Ann. Mat. Pura Appl. (4)}, 146:65--96, 1987.

\bibitem{T07}
A.~Tudorascu.
\newblock {Lubrication approximation for thin viscous films: asymptotic
  behavior of nonnegative solutions}.
\newblock {\em Comm. Partial Differential Equations}, 32(7-9):1147--1172, 2007.

\end{thebibliography}
\end{document}